\documentclass{amsart}
\usepackage{graphicx, color}
\usepackage{amssymb}
\newtheorem{thm}{Theorem}[section]
\newtheorem{cor}[thm]{Corollary}
\newtheorem{lem}[thm]{Lemma}
\newtheorem{prop}[thm]{Proposition}
\theoremstyle{definition}
\newtheorem{defn}[thm]{Definition}
\theoremstyle{remark}
\newtheorem{rem}[thm]{Remark}
\numberwithin{equation}{section}
\newcommand{\norm}[1]{\left\Vert#1\right\Vert}
\newcommand{\abs}[1]{\left\vert#1\right\vert}
\newcommand{\set}[1]{\left\{#1\right\}}

\newcommand{\eps}{\varepsilon}

\newcommand{\dbar}{\bar\partial}
\newcommand{\dbarstar}{\bar\partial^{\ast}}
\newcommand{\zbar}{\overline{z}}

\newcommand{\Htg}{\mathbb{H}_{\gamma}}
\newcommand{\BpHt}{P_{\gamma}}

\newcommand{\ddbar}{\partial\bar\partial}

\DeclareMathOperator{\dom}{Dom}

\DeclareMathOperator{\re}{Re}
\DeclareMathOperator{\im}{Im}

\DeclareMathOperator{\dist}{dist}

\DeclareMathOperator{\meas}{meas}

\begin{document}

\title[$L^p$ Mapping Properties on Lipschitz Domains]{$L^p$ Mapping Properties for the Cauchy-Riemann Equations on Lipschitz Domains Admitting Subelliptic Estimates}%
\author{Phillip S. Harrington}
\address{SCEN 309, 1 University of Arkansas, Fayetteville, AR 72701}%
\email{psharrin@uark.edu}%

\author{Yunus E. Zeytuncu}
\address{2014 CASL Building, University of Michigan-Dearborn, Dearborn, MI 48128}%
\email{zeytuncu@umich.edu}%

\thanks{The work of the second author was partially supported by a grant from the Simons Foundation (\#353525)}



\begin{abstract}
We show that on bounded Lipschitz pseudoconvex domains that admit good weight functions the $\dbar$-Neumann operators $N_q, \dbar^* N_{q}$, and $\dbar N_{q}$ are bounded on $L^p$ spaces for some values of $p$ greater than 2.
\end{abstract}
\maketitle



\section{Introduction}

Let $\Omega\subset \mathbb{C}^n$ be a bounded Lipschitz pseudoconvex domain.  Let $\dbar$ denote the Cauchy-Riemann operator acting on $(0,q)$-forms, and let $\dbarstar$ denote its adjoint with respect to the $L^2$ inner product.  For $1\leq q\leq n$, let $\Box_q:=\dbar\dbarstar+\dbarstar\dbar$ be the complex Laplacian, which is an unbounded self-adjoint operator on $L^2_{(0,q)}(\Omega)$. The self-adjoint bounded inverse of this operator is the $\dbar$-Neumann operator $N_q$. In the study of $\Box_q$ and $N_q$, two other operators $\dbar N_{q-1}:L^2_{(0,q-1)}(\Omega)\rightarrow L^2_{(0,q)}(\Omega)$ and $ \dbar^* N_{q}:L^2_{(0,q)}(\Omega)\rightarrow L^2_{(0,q-1)}(\Omega)$ naturally arise as they provide canonical solution operators to $\dbarstar$ and $\dbar$-problems. Another operator of interest in this setting is the Bergman projection $P_q:L^2_{(0,q)}(\Omega)\rightarrow Ker(\dbar)\cap L^2_{(0,q)}(\Omega)$. We refer to \cite{ChSh01} and \cite{Str10} for details about these operators and the setup.

In this paper we are interested in $L^p$ mapping properties of the $\dbar$-Neumann operators $N_q, \dbar^* N_{q}$, and $\dbar N_{q}$. In particular, we observe that these operators exhibit low-level $L^p$ regularity on bounded Lipschitz pseudoconvex domains that admit good weight functions. It is known that under more stringent geometric assumptions these operators are continuous on $L^p$ spaces for the full range of $p\in (1,\infty)$. For example, if $\Omega$ is a smoothly bounded pseudoconvex domain of finite type in $\mathbb{C}^n$ ($n\geq 3$) having comparable Levi eigenvalues, then it is shown in \cite[Theorem 1.1]{Koe04} that the $\dbar$-Neumann operators are not only bounded on corresponding $L^p$ spaces for all $p\in(1,\infty)$, but they gain derivatives in the Sobolev scale. Similar results were previously obtained on smooth strongly pseudoconvex domains in \cite{Cha88}, and on smooth finite type domains in $\mathbb{C}^2$ in \cite{CNS92}.

In contrast to these boundedness results on smooth domains, in some recent papers \cite{ZeytuncuTran, ChakZey, Chen14, EdM16, ChenZeytuncu} it is noted that on certain domains the Bergman projection $P_0$ may fail to be bounded on $L^p$ spaces for any $p\not=2$. Following the ideas in these papers one can also observe that the $\dbar$-Neumann operators may exhibit similarly degenerate $L^p$ mapping properties on domains with singular boundaries.

In order to illustrate this point more specifically, we look at the $\dbar$-Neumann operators  on the generalized Hartogs triangle $\mathbb{H}_{\gamma}$ defined by $$\Htg=\left\{(z_1,z_2)\in\mathbb{C}^2~:~|z_1|^{\gamma}<|z_2|<1\right\}$$
for a positive irrational parameter $\gamma>0$. Let $\BpHt$ denote the Bergman projection from $L^2(\Htg)$ onto $L^2(\Htg)\cap\ker\dbar$. In \cite{EdM16} Theorem 0.3, Edholm and McNeal show that $P_{\gamma}$ is bounded on $L^p(\Htg)$ if and only if $p=2$. Using their proof in Section 5, we will show that the canonical solution operator $\dbarstar N_1$ is bounded from $L^p_{(0,1)}(\Htg)$ to $L^p_{(0,0)}(\Htg)$ if and only if $p=2$.  Using a classical result of Dirichlet, Edholm and McNeal begin with a sequence of rational numbers $\set{\frac{m_j}{n_j}}$ converging to $\gamma$ such that $\abs{\frac{n_j}{m_j}-\frac{1}{\gamma}}<\frac{1}{m_j^2}$.  Using this, they show that there exists an integer $0\leq\beta_1\leq m_j-1$ such that  $\beta_2=\frac{1-n_j\beta_1-n_j-m_j}{m_j}$ is also an integer and $z_1^{\beta_1} z_2^{\beta_2}\in L^2(\Htg)$.  Since $n_j$ is a positive integer, $n_j\geq 1$, so $\beta_2<-1$.  Set $f_j(z_1,z_2)=\frac{z_1^{\beta_1}}{\zbar_2^{\beta_2}}$. Then $\dbar f_j(z)$ is a $(0,1)$ form with bounded coefficients and consequently $\dbar f_j\in L^p_{(0,1)}(\Htg)$ for all $p\geq 1$. Furthermore, using Proposition 4.1 in \cite{EdM16}, we see that $\BpHt(f_j)=c_jz_1^{\beta_1}z_2^{\beta_2}$ for some constant $c_j$.  Edholm and McNeal then show that for any $p_0>2$ one can choose $j$ large enough so that $z_1^{\beta_1}z_2^{\beta_2}\not \in L^{p_0}(\Htg)$. On the other hand, we note that $\dbarstar N_1(\dbar f_j)=f_j-\BpHt(f_j)$. Hence $\dbarstar N_1$ is not bounded from $L^p_{(0,1)}(\Htg)$ to $L^p_{(0,0)}(\Htg)$ when $p>2$.

If we take a smooth compactly supported function $\chi$ on $\Htg$ that is radially symmetric, i.e., $\chi(z_1,z_2)=\chi(|z_1|,|z_2|)$, then the proof of Proposition 4.1 in \cite{EdM16} can be used to show that $\BpHt(\chi f_j)=\widehat{c}_jz_1^{\beta_1}z_2^{\beta_2}$ for some other constant $\widehat{c}_j$.  Since $\chi f_j$ is compactly supported in $\Htg$, $\dbar\left(\chi f_j\right)\in\dom(\dbarstar)$, so we can use the decomposition in \cite[Theorem 4.4.3]{ChSh01} to obtain $N_0(\dbarstar\dbar(\chi f_j))=\chi f_j-\BpHt(\chi f_j)$. Again, $\dbarstar\dbar(\chi f_j)\in  L^p(\Htg)$ for all $p\geq 1$, but for any $p_0>2$ one can choose $j$ large enough such that $N_0(\dbarstar\dbar \chi f_j) \not \in L^{p_0}(\Htg)$. Hence $N_0$ is not bounded from $L^p_{(0,0)}(\Htg)$ to $L^p_{(0,0)}(\Htg)$ when $p>2$. Combining the calculations for $\dbar^* N_1$ with a duality argument, one can also show  $\dbar N_0$ is not bounded from $L^p_{(0,0)}(\Htg)$ to $L^p_{(0,1)}(\Htg)$ when $1<p<2$. In other words, on the generalized Hartogs triangle $\Htg$ the $\dbar$-Neumann operators $N_q, \dbar^* N_{q}$, and $\dbar N_{q}$ exhibit degenerate $L^p$ mapping properties.

When we compare the $L^p$ boundedness results for the full range of $p\in (1,\infty)$ on smooth domains and this peculiar degenerate $L^p$ regularity on domains like the generalized Hartogs triangles, a natural question arises: is it possible to obtain low level $L^p$ regularity on domains with Lipschitz boundaries? The similar question on low level Sobolev regularity was previously answered in \cite{Koh99} and \cite{BeCh00}. Initial work on $L^p$ estimates on domains with low boundary regularity was carried out in \cite{BoCh90}.  In this paper, we provide an affirmative answer to this question and establish the following low level $L^p$ regularity results.

\begin{thm}
\label{thm:main_theorem}
  Let $\Omega\subset\mathbb{C}^n$ be a bounded Lipschitz pseudoconvex domain and suppose that there exist constants $0<\eps\leq\frac{1}{2}$ and $C>0$ and a continuous function $\lambda$ satisfying $0\leq\lambda\leq 1$ and
  \[
    i\ddbar\lambda\geq iC(\delta(z))^{-2\eps}\ddbar\abs{z}^2
  \]
  on $\Omega$ in the sense of currents.  Then for any $p$ satisfying
  \begin{equation}
  \label{eq:p_range}
    2<p<\frac{2n}{n-\eps}.
  \end{equation}
  the $\dbar$-Neumann operators
  \begin{align*}
    N_q:&L^{p/(p-1)}_{(0,q)}(\Omega)\rightarrow L^p_{(0,q)}(\Omega)\\
    \dbar N_{q-1}:&L^2_{(0,q-1)}(\Omega)\rightarrow L^p_{(0,q)}(\Omega)
  \end{align*}
  are continuous for all $1\leq q\leq n$ and
  \[
    \dbar^* N_{q}:L^2_{(0,q)}(\Omega)\rightarrow L^p_{(0,q-1)}(\Omega).
  \]
  is continuous for all $2\leq q\leq n$.  If, in addition, there exists a defining function $\rho$ for $\Omega$ and a constant $0<\eta<1$ such that $-(-\rho)^\eta$ is plurisubharmonic on $\Omega$, then for any $p$ satisfying
  \begin{equation}
  \label{eq:p_range_2}
    2<p<\min\left\{\frac{8n^2}{4n^2-\eta}, \frac{4n^2}{2n^2-\eps}\right\},
  \end{equation}
  the solution operator
  \[
    \dbar^* N_1:L^2_{(0,1)}(\Omega)\rightarrow L^p(\Omega)
  \]
  is continuous.
\end{thm}

\begin{cor}
\label{cor:finite_type}
  Let $\Omega\subset\mathbb{C}^n$ be a bounded domain that is locally a transverse intersection of smooth finite type domains.  Then there exists some $p_0>2$ such that the $\dbar$-Neumann operators
  \begin{align*}
    N_q:&L^{p/(p-1)}_{(0,q)}(\Omega)\rightarrow L^p_{(0,q)}(\Omega)\\
    \dbar^* N_{q}:&L^2_{(0,q)}(\Omega)\rightarrow L^p_{(0,q-1)}(\Omega)\\
    \dbar N_{q-1}:&L^2_{(0,q-1)}(\Omega)\rightarrow L^p_{(0,q)}(\Omega)
  \end{align*}
  are continuous for all $1\leq q\leq n$ and $2<p<p_0$.
\end{cor}

As suggested by the hypotheses of Corollary \ref{cor:finite_type}, our ideal result would state that the $\dbar$-Neumann operators have good $L^p$ mapping properties on finite type domains.  Unfortunately, it is not clear how to define finite type on Lipschitz domains.  Catlin's work in \cite{Cat84} and \cite{Cat87} demonstrated that on smooth domains, his finite type condition is equivalent to subelliptic estimates and the existence of a family of good weight functions.  Since Straube showed in \cite{Str97} that the existence of a family of good weight functions also implies subelliptic estimates on Lipschitz domains, these conditions seem to be the most natural substitute for finite type on Lipschitz domains. Hence, we adopt these as our hypotheses in the statement of Theorem \ref{thm:main_theorem}.  Following Catlin's work and Straube's extension to Lipschitz domains in \cite{Str97}, the finite type assumption in Corollary \ref{cor:finite_type} implies the existence of the weight functions that are needed in Theorem \ref{thm:main_theorem}, which in turn imply the subelliptic estimates that we use in the proof of Theorem \ref{thm:main_theorem}.

We remark that there are $L^p$ regularity results for solution operators (not the canonical solution operator) to the $\dbar$-problem on some more general domains, see \cite{ChL2000, Ker71, Ovr71}. In these studies, the solution operators are integral operators with a Calderon-Zygmund type kernel, and the theory of singular integral operators are employed to obtain regularity results. As a complementary result in this approach, it was noted in \cite{FoS89} that there are certain domains on which there exists no solution operator that is bounded on $L^p$ spaces for the full range of $p\in (1,\infty)$.

The authors would like to thank the referee for many helpful suggestions that have significantly improved this paper.

\section{Related Norms}

In this section, we will prove relationships between $L^p$ norms and weighted $L^2$ norms.  Closely related results were obtained in \cite{BoSi91} and \cite{FoS89}, but we will prove versions that are adapted to our methods of proof.  For a domain $\Omega\subset\mathbb{C}^n$, let $\delta(z)=\dist(z,\partial\Omega)$.  Let $\omega_{2n}$ denote the volume of the unit ball in $\mathbb{C}^n$.  A domain $\Omega\subset\mathbb{C}^n$ is said to be a Lipschitz domain if $\partial\Omega$ is locally the graph of a Lipschitz function.

The estimates in Lemmas \ref{lem:weighted_norm_bound} and \ref{lem:Lp_norm_bound} will both be actual estimates (the left-hand side is finite whenever the right-hand side is finite) as opposed to a priori estimates (the estimate only holds if we first assume that both sides are finite).  This will be immediate for Lemma \ref{lem:weighted_norm_bound} because H\"older's inequality is an actual estimate, but proving that the estimate in Lemma \ref{lem:Lp_norm_bound} is an actual estimate will require more caution.

\begin{lem}
\label{lem:weighted_volume_estimate}
  Let $\Omega\subset\mathbb{C}^n$ be a bounded Lipschitz domain.  Then there exists $C>0$ depending only on $\Omega$ such that for every $q\geq 1$ and $0\leq r<\frac{1}{q}$ we have
  \begin{equation}
  \label{eq:weighted_volume_estimate}
    \norm{\delta^{-r}}_{L^q(\Omega)}\leq C(1-rq)^{-1/q}
  \end{equation}
\end{lem}

\begin{proof}
  For each point in $\partial\Omega$ there exists a neighborhood $U$ with orthonormal coordinates on which
  \[
    \Omega\cap U=\set{z\in U:\im z_n<\phi(z',\re z_n)},
  \]
  where $z'=(z_1,\ldots,z_{n-1})$ and $\phi$ is a Lipschitz function with Lipschitz constant $M$.  For $z\in\Omega\cap U$, the cone
  \[
    \Gamma_z=\set{w\in U:\im w_n<\phi(z',\re z_n)-M\sqrt{\abs{w'-z'}^2+(\re z_n-\re w_n)^2}}
  \]
  satisfies $z\in\Gamma_z\subset\Omega\cap U$, so $\delta(z)\geq\dist(z,\partial\Gamma_z)$.  If we abbreviate
  \[
    x=\sqrt{\abs{w'-z'}^2+(\re z_n-\re w_n)^2},
  \]
  then
  \[
    \dist(z,\partial\Gamma_z)=\inf_{x\geq 0}\sqrt{x^2+(\im z_n-\phi(z',\re z_n)+Mx)^2},
  \]
  which is optimized when $x=\frac{M}{1+M^2}(\phi(z',\re z_n)-\im z_n)$, so
  \[
    \delta(z)\geq\dist(z,\partial\Gamma_z)=\frac{1}{\sqrt{1+M^2}}(\phi(z',\re z_n)-\im z_n).
  \]
  Therefore,
  \[
    \int_{\Omega\cap U}\delta^{-rq}dV\leq(1+M^2)^{rq/2}\int_{\Omega\cap U}(\phi(z',\re z_n)-\im z_n)^{-rq}dV.
  \]
  By shrinking $U$ as needed, we may assume $U=\tilde U\times(-\eps,\eps)$ where $\tilde U\subset\mathbb{C}^{n-1}\times\mathbb{R}$ with $\meas(\tilde U)\leq 1$ and $0<\eps\leq\frac{1}{2}$.  Thus,
  \begin{multline*}
    \int_{\Omega\cap U}\delta^{-rq}dV\leq(1+M^2)^{rq/2}\meas(\tilde U)\int_0^{2\eps}t^{-rq}dt\\
    =(1+M^2)^{rq/2}\meas(\tilde U)(1-rq)^{-1}(2\eps)^{1-rq}<(1+M^2)^{rq/2}(1-rq)^{-1}.
  \end{multline*}
  Since $\partial\Omega$ is compact, we may cover $\partial\Omega$ with finitely many such $U$.  Using the finiteness of the cover and $r<1$, we may find an upper bound independent of $q$ and $r$ for $(1+M^2)^{r/2}$.  Via a partition of unity, there must exist a constant $C>0$ such that \eqref{eq:weighted_volume_estimate} holds.
\end{proof}

\begin{lem}
\label{lem:weighted_norm_bound}
  Let $\Omega\subset\mathbb{C}^n$ be a bounded Lipschitz domain, and let $u\in L^1_{loc}(\Omega)$.  Then there exists $C>0$ such that for every $p>2$ and $0\leq s<\frac{p-2}{2p}$ we have
  \begin{equation}
  \label{eq:weighted_norm_bound_1}
    \norm{\delta^{-s}u}_{L^2(\Omega)}\leq \sqrt{C}\left(1-\frac{2s p}{p-2}\right)^{-(p-2)/{2p}}\norm{u}_{L^p(\Omega)},
  \end{equation}
  and for every $0<p<2$ and $0\leq s<\frac{2-p}{2p}$ we have
  \begin{equation}
  \label{eq:weighted_norm_bound_2}
    \norm{u}_{L^p(\Omega)}\leq C^{1/p}\left(1-\frac{2s p}{2-p}\right)^{-(2-p)/{2p}}\norm{\delta^s u}_{L^2(\Omega)}.
  \end{equation}
\end{lem}

\begin{rem}
  Note that \eqref{eq:weighted_norm_bound_1} has a natural extension to the $p=\infty$ case which follows immediately from \eqref{eq:weighted_volume_estimate}, but we will not need to make use of this.
\end{rem}

\begin{proof}
  We first suppose $p>2$.  We may assume $u\in L^p(\Omega)$.  By H\"older's inequality,
  \[
    \int_\Omega\abs{u}^2\delta^{-2s}dV\leq\norm{\abs{u}^2}_{L^{p/2}(\Omega)}\norm{\delta^{-2s}}_{L^{p/(p-2)}(\Omega)}.
  \]
  Since $\frac{p}{p-2}>1$, we may use \eqref{eq:weighted_volume_estimate} to obtain \eqref{eq:weighted_norm_bound_1}.

  Turning to the case where $0<p<2$, we may assume $\delta^s u\in L^2(\Omega)$.  H\"older's inequality gives us
  \[
    \int_\Omega|u|^p dV\leq\norm{|u|^p\delta^{ps}}_{L^{2/p}(\Omega)}\norm{\delta^{-ps}}_{L^{\frac{2}{2-p}}(\Omega)}
  \]
  Since $\frac{2}{2-p}>1$, we may use \eqref{eq:weighted_volume_estimate} to obtain \eqref{eq:weighted_norm_bound_2}.
\end{proof}

\begin{lem}
\label{lem:Lp_norm_bound}
  Let $\Omega\subset\mathbb{C}^n$ be a bounded Lipschitz domain, and let $u$ be a harmonic function on $\Omega$.  Then for any $p>2$,
  \begin{equation}
  \label{eq:Lp_norm_bound_1}
    \norm{u}_{L^p(\Omega)}\leq\frac{1}{\omega_{2n}^{(p-2)/2p}}\norm{\delta^{-n(p-2)/p}u}_{L^2(\Omega)},
  \end{equation}
  and for any $1\leq p<2$,
  \begin{equation}
  \label{eq:Lp_norm_bound_2}
    \norm{\delta^{n(2-p)/p}u}_{L^2(\Omega)}\leq\frac{1}{\omega_{2n}^{(2-p)/{2p}}}\norm{u}_{L^p(\Omega)}.
  \end{equation}
\end{lem}

\begin{rem}
  Once again, we note that \eqref{eq:Lp_norm_bound_1} has a natural generalization to the $p=\infty$ case which will follow immediately by substituting $p=2$ in \eqref{eq:mean_value_property_consequence}, but we will not need this in our main results.
\end{rem}

\begin{proof}
  For $\eps>0$, set $\Omega_\eps=\set{z\in\Omega:\delta(z)>\eps}$.  Fix $z\in\Omega_\eps$, and choose $r<\delta(z)-\eps$.  By the mean value property,
  \[
    u(z)=\frac{1}{\omega_{2n}r^{2n}}\int_{B(z,r)}u(w)dV_w.
  \]
  Since $x^p$ is a convex function when $p\geq 1$ and $x\geq 0$, we may use Jensen's inequality to obtain
  \begin{equation}
  \label{eq:Jensens}
    \abs{u(z)}^p\leq\frac{1}{\omega_{2n} r^{2n}}\int_{B(z,r)}|u(w)|^p dV_w.
  \end{equation}
  Since this holds for every $r<\delta(z)-\eps$, we must have
  \begin{equation}
  \label{eq:mean_value_property_consequence}
    \abs{u(z)}\leq\frac{1}{\omega_{2n}^{1/p} (\delta(z)-\eps)^{2n/p}}\norm{u}_{L^p(\Omega_\eps)}
  \end{equation}
  for any $z\in\Omega_\eps$.

 To prove \eqref{eq:Lp_norm_bound_1}, we first fix $p>2$.  Let $s=\frac{n(p-2)}{p}>0$.  For $z\in\Omega_\eps$ and $r<\delta(z)-\eps$, we note that $\delta(w)\leq\delta(z)+r$ for $w\in B(z,r)$, so \eqref{eq:Jensens} gives us
  \[
    \abs{u(z)}^2\leq\frac{1}{\omega_{2n} r^{2n}}\int_{B(z,r)}\left(\frac{\delta(z)+r}{\delta(w)}\right)^{2s}|u(w)|^2 dV_w\leq\frac{(\delta(z)+r)^{2s}}{\omega_{2n} r^{2n}}\norm{\delta^{-s}u}^2_{L^2(\Omega_\eps)}.
  \]
  Since this holds for every $r<\delta(z)-\eps$, we must have
  \[
    \abs{u(z)}\leq\frac{(2\delta(z)-\eps)^s}{\omega_{2n}^{1/2} (\delta(z)-\eps)^n}\norm{\delta^{-s}u}_{L^2(\Omega_\eps)}
  \]
  for any $z\in\Omega_\eps$.  We may raise this to the power of $p-2$ and substitute as follows:
  \begin{multline*}
    \norm{u}_{L^p(\Omega_{2\eps})}^p=\int_{\Omega_{2\eps}}\abs{(\delta(z))^{-s}u(z)}^2(\delta(z))^{2s}\abs{u(z)}^{p-2}dV\\
    \leq\int_{\Omega_{2\eps}}\abs{(\delta(z))^{-s}u(z)}^2\frac{(\delta(z))^{2s}(2\delta(z)-\eps)^{s(p-2)}}{\omega_{2n}^{(p-2)/2} (\delta(z)-\eps)^{n(p-2)}}\norm{\delta^{-s}u}^{p-2}_{L^2(\Omega_\eps)}dV.
  \end{multline*}
  Using elementary calculus, one can check that $\frac{\delta^{2}(2\delta-\eps)^{p-2}}{(\delta-\eps)^p}$ is a decreasing function with respect to $\delta$ when $\delta>\eps$, and hence $\frac{(\delta(z))^{2}(2\delta(z)-\eps)^{p-2}}{(\delta(z)-\eps)^p}\leq4\cdot 3^{p-2}$ on $\Omega_{2\eps}$.  Since $ps=n(p-2)$ by definition, we have
  \[
    \norm{u}_{L^p(\Omega_{2\eps})}^p
    \leq \frac{4^s 3^{s(p-2)}}{\omega_{2n}^{(p-2)/2}}\norm{\delta^{-s}u}^2_{L^2(\Omega_{2\eps})}\norm{\delta^{-s}u}^{p-2}_{L^2(\Omega_\eps)}.
  \]
  If $\norm{\delta^{-s}u}_{L^2(\Omega)}=\infty$ then \eqref{eq:Lp_norm_bound_1} is trivial, so we assume this quantity is finite.  Then
  \[
    \norm{\delta^{-s}u}^p_{L^2(\Omega\backslash\Omega_\eps)}\rightarrow 0\text{ and }\norm{\delta^{-s}u}^p_{L^2(\Omega\backslash\Omega_{2\eps})}\rightarrow 0,
  \]
  so
  \[
    \norm{\delta^{-s}u}_{L^2(\Omega_\eps)}\rightarrow\norm{\delta^{-s}u}_{L^2(\Omega)}\text{ and }\norm{\delta^{-s}u}_{L^2(\Omega_{2\eps})}\rightarrow\norm{\delta^{-s}u}_{L^2(\Omega)}.
  \]
  If we define $u_\eps$ to equal $u$ on $\Omega_{2\eps}$ and $0$ on $\Omega\backslash\Omega_{2\eps}$, then we have
  \[
    \norm{u_\eps}_{L^p(\Omega)}
    \leq \frac{4^{s/p} 3^{s(p-2)/p}}{\omega_{2n}^{(p-2)/{2p}}}\norm{\delta^{-s}u}^{2/p}_{L^2(\Omega_{2\eps})}\norm{\delta^{-s}u}^{(p-2)/p}_{L^2(\Omega_\eps)}.
  \]
  so we may use Fatou's lemma to conclude $u\in L^p(\Omega)$.

  Now that we know that $u\in L^p(\Omega)$ whenever $\norm{\delta^{-n(p-2)/p}u}_{L^2(\Omega)}<\infty$, we may show that \eqref{eq:mean_value_property_consequence} holds with $\eps=0$.  We next raise \eqref{eq:mean_value_property_consequence} to the power of $p-2$ and substitute in the following inequality:
  \[
    \norm{u}_{L^p(\Omega)}^p=\int_{\Omega}\abs{u}^2\cdot\abs{u}^{p-2}dV\leq\frac{1}{\omega_{2n}^{(p-2)/p}}\norm{u}_{L^p(\Omega)}^{p-2}\norm{\delta^{-n(p-2)/p}u}_{L^2(\Omega)}^2.
  \]
  Dividing by $\norm{u}_{L^p(\Omega)}^{p-2}$ (which we now know to be finite) and taking a square root will give us \eqref{eq:Lp_norm_bound_1}.

  To prove \eqref{eq:Lp_norm_bound_2}, we fix $1\leq p<2$.  This time, we raise \eqref{eq:mean_value_property_consequence} to the power of $2-p$ and substitute to obtain
  \begin{multline*}
    \norm{u(\delta-\eps)^{n(2-p)/p}}_{L^2(\Omega_\eps)}^2=\int_{\Omega_\eps}\abs{u}^p\cdot\abs{u}^{2-p}(\delta-\eps)^{2n(2-p)/p}dV\\
    \leq\frac{1}{\omega_{2n}^{(2-p)/p}}\norm{u}_{L^p(\Omega_\eps)}^2.
  \end{multline*}
  Once again, the result is trivial if $u\notin L^p(\Omega)$, and the right-hand side converges when $u\in L^p(\Omega)$, so Fatou's lemma may be used to prove \eqref{eq:Lp_norm_bound_2}.
\end{proof}

\section{Good Weight Functions}

The connection between finite type and good weight functions was first observed by Catlin in \cite{Cat84} and \cite{Cat87}.  Straube observed that Catlin's result could be used to construct useful weight functions on certain Lipschitz domains in \cite{Str97}.  Although finite type is essentially a local condition, we will need good global weight functions, so we begin by showing that the functions constructed by Catlin and Straube can be patched together.

\begin{lem}
\label{lem:finite_type_weight}
  Let $\Omega\subset\mathbb{C}^n$ be a bounded Lipschitz pseudoconvex domain that is locally a transverse intersection of smooth finite type domains.  Then there exist constants $\frac{1}{2}\geq\eps>0$ and $C>0$ and a function $\lambda$ such that
  \[
    i\ddbar\lambda\geq C(\delta(z))^{-2\eps}i\ddbar\abs{z}^2
  \]
  on $\Omega$ in the sense of currents and $0\leq\lambda\leq 1$ on $\Omega$.
\end{lem}

\begin{proof}
  For each point $p\in\partial\Omega$, there exists a neighborhood $B(p,r_p)$ of $p$ such that the result holds on $\Omega\cap B(p,r_p)$ for constants $\eps_p$ and $C_p$ and a function $\lambda_p$ by an argument of Straube (see the proof of Theorem 2 in \cite{Str97}).  Without loss of generality, we may assume $0\leq\lambda_p\leq 1$ on $B(p,r_p)$.

  Since $\partial\Omega$ is compact, there is a finite set $\mathcal{P}\subset\partial\Omega$ such that $B(p,r_p/3)$ covers $\partial\Omega$ for $p\in\mathcal{P}$.  Fix $\delta_1>0$ so that $\{B(p,r_p/3)\}_{p\in\mathcal{P}}$ covers $K_1=\{z\in\overline\Omega:\delta(z)\leq\delta_1\}$.  On $K_1$, let
  \[
    \lambda_1(z)=\sup_{\{p\in\mathcal{P}:z\in B(p,r_p)\}}\frac{9}{14}+\frac{5}{14}\lambda_p-\frac{18}{7} r_p^{-2}\abs{z-p}^2.
  \]
  Note that when $z\in B(p,r_p/3)$,
  \[
    \frac{9}{14}+\frac{5}{14}\lambda_p-\frac{18}{7} r_p^{-2}\abs{z-p}^2\geq\frac{9}{14}-\frac{2}{7}=\frac{5}{14}
  \]
  and when $z\notin B(p,r_p/2)$,
  \[
    \frac{9}{14}+\frac{5}{14}\lambda_p-\frac{18}{7} r_p^{-2}\abs{z-p}^2\leq\frac{5}{14},
  \]
  so the supremum must be obtained for some $p\in\mathcal{P}$ satisfying $z\in B(p,r_p/2)$.  Hence $\lambda_1$ is continuous on $K_1$ and satisfies $0\leq\lambda\leq 1$ on $K_1$.

  In the sense of currents, we have
  \[
    i\ddbar\left(\frac{9}{14}+\frac{5}{14}\lambda_p-\frac{18}{7} r_p^{-2}\abs{z-p}^2\right)\geq\left(\frac{5}{14}C_p(\delta(z))^{-2\eps_p}-\frac{18}{7} r_p^{-2}\right)i\ddbar\abs{z}^2
  \]
  on $B(p,r_p)\cap\Omega$.  If we set $\eps=\inf_{p\in\mathcal{P}}\eps_p$, choose $\delta_2<\inf_{p\in\mathcal{P}}\left(\frac{5}{36}C_p r_p^{2}\right)^{1/(2\eps)}$, and set $C_1=\inf_{p\in\mathcal{P}}\left(\frac{5}{14}C_p-\frac{18}{7} r_p^{-2}(\delta_2)^{2\eps}\right)$, then \[
    i\ddbar\left(\frac{9}{14}+\frac{5}{14}\lambda_p-\frac{18}{7} r_p^{-2}\abs{z-p}^2\right)\geq C_1(\delta(z))^{-2\eps}i\ddbar\abs{z}^2.
  \]
  in the sense of currents on $B(p,r_p)\cap K_2$, where $K_2=\{z\in K_1:\delta(z)\leq\delta_2\}$.  We may conclude that
  \[
    i\ddbar\lambda_1\geq C_1(\delta(z))^{-2\eps_p}i\ddbar\abs{z}^2
  \]
  on $K_2$.

  To extend into the interior, we let $\mu$ be a bounded strictly plurisubharmonic exhaustion function for $\Omega$.  Such a function exists by a result of Demailly \cite{Dem87} (see also \cite{Har08a} for a refinement of this result).  After rescaling, we may assume $\mu=0$ on $\partial\Omega$ and $-1=\sup_{\Omega\backslash K_2}\mu$.  Set $A=-\inf_{K_2}\mu(z)$, $B=-\inf_{\overline\Omega}\mu(z)$, and choose constants $(A+1)^{-1}<D<\min\{A^{-1},1\}$, $1-D<F<AD$, and $0<E<\min\{D+F-1,B^{-1}F\}$.  Define
  \[
    \lambda(z)=\begin{cases}\max\{(1-AD)\lambda_1(z)+D\mu(z)+AD,E\mu(z)+F\}&z\in K_2\\E\mu(z)+F&z\in\Omega\backslash K_2\end{cases}.
  \]
  On $\partial\Omega$, $(1-AD)\lambda_1(z)+D\mu(z)+AD\geq AD$ and $E\mu(z)+F\leq F$, so since $AD>F$,
  \[
    i\ddbar\lambda\geq (1-AD)C_1(\delta(z))^{-2\eps_p}i\ddbar\abs{z}^2
  \]
  on some interior neighborhood of $\partial\Omega$ in the sense of currents.  When $\delta(z)=\delta_2$, $(1-AD)\lambda_1(z)+D\mu(z)+AD<1+D\mu(z)$.  Note that since $1>AD>F$ and $D+F-1>E$, we must have $D-E>1-F>0$.  Since $\mu(z)\leq -1$ when $\delta(z)=\delta_2$ by the maximum principle, $1-D<F-E$, and $D>E$, we must have $1+D\mu(z)<E\mu(z)+F$ at such points, and hence $\lambda(z)=E\mu(z)+F$ in a neighborhood of $\Omega\backslash K_2$.  From this we conclude that $\lambda$ is continuous on $\Omega$.  On $K_2$, we note that $(1-AD)\lambda_1(z)+D\mu(z)+AD\leq 1$ and $(1-AD)\lambda_1(z)+D\mu(z)+AD\geq 0$, while on $\Omega$ we have $E\mu(z)+F\leq F<1$ and $E\mu(z)+F\geq -EB+F>0$, so $0\leq\lambda\leq 1$ on $\Omega$.  Since $\mu$ is strictly plurisubharmonic, we may find $C>0$ so that our conclusion holds.
\end{proof}

\begin{defn}
  Let $\Omega\subset\mathbb{C}^n$ be a Lipschitz pseudoconvex domain.  The Diederich-Fornaess Index of $\Omega$ is defined to be the supremum over all exponents $0<\eta<1$ such that there exists a defining function $\rho$ for $\Omega$ such that $-(-\rho)^\eta$ is plurisubharmonic on $\Omega$. We denote this $DF(\Omega)$.
\end{defn}

The Diederich-Fornaess Index for a Lipschitz pseudoconvex domain is strictly positive, as shown in \cite{Har08a}.  On a $C^2$ domain admitting a family of good weight functions such as those constructed in \cite{Cat84}, it is known that the Diederich-Fornaess Index is equal to one (see the opening paragraph of Section 5 in \cite{KLP16}).  The following result is a slight simplification of a result first proved in \cite{Har07b}.

\begin{lem}
\label{lem:diederich_fornaess_weight}
  Let $\Omega\subset\mathbb{C}^n$ be a bounded Lipschitz pseudoconvex domain and suppose that there exist constants $0<\eps\leq\frac{1}{2}$ and $C>0$ and a continuous function $\lambda$ satisfying $0\leq\lambda\leq 1$ and
  \[
    i\ddbar\lambda\geq iC(\delta(z))^{-2\eps}\ddbar\abs{z}^2
  \]
  on $\Omega$ in the sense of currents.  Then for every $0\leq\eta<DF(\Omega)$ there exists a continuous function $\mu$ such that $k_\eta(\delta(z))^\eta\geq-\mu(z)\geq (\delta(z))^\eta$ on $\Omega$ for some $k_\eta>1$ and
  \begin{equation}
  \label{eq:mu_hessian}
    i\ddbar\mu(z)\geq i\frac{(DF(\Omega)-\eta)C}{2 e DF(\Omega)}(-\mu(z))(\delta(z))^{-2\eps}\ddbar\abs{z}^2
  \end{equation}
  on $\Omega$ in the sense of currents.  Furthermore, $k_\eta$ is uniformly bounded on any compact subset of $[0,DF(\Omega))$.
\end{lem}

\begin{rem}
  We have chosen to work with a bounded weight $\lambda$ because the weight obtained from Catlin and Straube's work will be bounded, but it is worth noting that the following proof is easily adapted to work with weights of self-bounded gradient.  See \cite{Her07} for a discussion of such weights in the context of subelliptic estimates.
\end{rem}

\begin{proof}
  Fix $\eta<s<DF(\Omega)$ and let $r$ be a defining function for $\Omega$ such that $-(-r)^s$ is plurisubharmonic on $\Omega$.  By rescaling we may assume $-r(z)\geq\delta(z)$.  We set
  \[
    \mu(z)=-(-r(z))^\eta\exp\left(\frac{s-\eta}{s}(1-e^{\lambda-1})\right).
  \]
  If $-r\leq k\delta$ for some $k>1$, then
  \[
    (\delta(z))^\eta\leq-\mu(z)\leq k^\eta(\delta(z))^\eta\exp\left(\frac{s-\eta}{s}(1-e^{-1})\right).
  \]
  Since $\frac{s-\eta}{s}<1$, we can take $k_\eta\leq k^\eta\exp(1-e^{-1})$, which is uniformly bounded for all $\eta<s$.  If we restrict to a compact subset of $\Omega$, we may regularize $-(-r)^s$ and $\lambda$ by convolution to obtain smooth functions with the same plurisubharmonicity properties.  Hence we may carry out the following computation under the assumption that $r$ and $\lambda$ are smooth, with the understanding that we will ultimately take limits to recover the necessary information about $\mu$.

  We first note that $-(-r)^s$ is plurisubharmonic if and only if
  \begin{equation}
  \label{eq:diederich_fornaess_exponent}
    i\ddbar r\geq -i(1-s)(-r)^{-1}\partial r\wedge\dbar r.
  \end{equation}
  To facilitate derivatives of $\mu$, we first take $-\log(-\mu)$ and compute
  \begin{equation}
  \label{eq:mu_derivative}
    (-\mu)^{-1}\dbar\mu=\eta(-r)^{-1}\dbar r+\frac{s-\eta}{s}e^{\lambda-1}\dbar\lambda
  \end{equation}
  and
  \begin{multline*}
    i(-\mu)^{-1}\ddbar\mu+i(-\mu)^{-2}\partial\mu\wedge\dbar\mu=i\eta(-r)^{-1}\ddbar r+i\eta(-r)^{-2}\partial r\wedge\dbar r\\
    +i\frac{s-\eta}{s}e^{\lambda-1}\ddbar\lambda+i\frac{s-\eta}{s}e^{\lambda-1}\partial\lambda\wedge\dbar\lambda.
  \end{multline*}
  Substituting \eqref{eq:diederich_fornaess_exponent} and \eqref{eq:mu_derivative}, we may simplify to obtain
  \begin{multline*}
    i(-\mu)^{-1}\ddbar\mu \geq  i\eta (s-\eta)(-r)^{-2}\partial r\wedge\dbar r-i\eta(-r)^{-1}\frac{s-\eta}{s}e^{\lambda-1}(\partial r\wedge\dbar\lambda+\partial\lambda\wedge\dbar r)\\
    i\frac{s-\eta}{s}e^{\lambda-1}\ddbar\lambda+i\frac{s-\eta}{s}e^{\lambda-1}\left(1-\frac{s-\eta}{s}e^{\lambda-1}\right)\partial\lambda\wedge\dbar\lambda.
  \end{multline*}
  Since
  \begin{multline*}
    \eta (s-\eta)(-r)^{-2}\frac{s-\eta}{s}e^{\lambda-1}\left(1-\frac{s-\eta}{s}e^{\lambda-1}\right)-\left(\eta(-r)^{-1}\frac{s-\eta}{s}e^{\lambda-1}\right)^2 =\\
    \eta (s-\eta)(-r)^{-2}\frac{s-\eta}{s}e^{\lambda-1}\left(1-e^{\lambda-1}\right)\geq0,
  \end{multline*}
  the terms involving $\partial r$ and $\partial\lambda$ must add up to a positive semi-definite form, so
  \[
    i(-\mu)^{-1}\ddbar\mu\geq i\frac{s-\eta}{s}e^{\lambda-1}\ddbar\lambda.
  \]
  If we set $s=\frac{2\eta DF(\Omega)}{DF(\Omega)+\eta}$, then $\frac{s-\eta}{s}=\frac{DF(\Omega)-\eta}{2DF(\Omega)}$, so our hypothesis on $\lambda$ can be use to obtain \eqref{eq:mu_hessian}.

\end{proof}

\section{The Basic Estimate}

We first recall a key result from Berndtsson and Charpentier \cite{BeCh00}:
\begin{thm}
\label{thm:BeCh00}
  Let $\Omega\subset\mathbb{C}^n$ be a bounded Lipschitz pseudoconvex domain.  Then the Bergman Projection $P_q$ is continuous in $L^2_{(0,q)}(\Omega,\delta^{-2s})$ for any $0\leq s<\frac{DF(\Omega)}{2}$.
\end{thm}

When $q=0$, the range of the Bergman Projection is holomorphic, and hence harmonic, so we may use this theorem with \eqref{eq:Lp_norm_bound_1} and \eqref{eq:weighted_norm_bound_1} to show
\[
  \norm{P_0 u}_{L^p(\Omega)}\lesssim\norm{\delta^{-n(p-2)/p}P_0 u}_{L^2(\Omega)}\lesssim\norm{\delta^{-n(p-2)/p}u}_{L^2(\Omega)}\lesssim\norm{u}_{L^{p'}(\Omega)}
\]
provided $0<\frac{n(p-2)}{p}<\min\left\{\frac{DF(\Omega)}{2},\frac{p'-2}{2p'}\right\}$.  Since $P_0$ is self-adjoint, we also have the dual estimate
\begin{equation}
\label{eq:Bergman_projection_estimate}
  \norm{P_0 u}_{L^{p'/(p'-1)}(\Omega)}\lesssim\norm{u}_{L^{p/(p-1)}(\Omega)}
\end{equation}
whenever $u\in L^{p/(p-1)}(\Omega)$ and $0<\frac{n(p-2)}{p}<\min\left\{\frac{DF(\Omega)}{2},\frac{p'-2}{2p'}\right\}$.

\begin{prop}
\label{prop:basic_estimate}
  Let $\Omega\subset\mathbb{C}^n$ be a bounded Lipschitz pseudoconvex domain and suppose that there exist constants $0<\eps\leq\frac{1}{2}$ and $C>0$ and a continuous function $\lambda$ satisfying $0\leq\lambda\leq 1$ and
  \[
    i\ddbar\lambda\geq iC(\delta(z))^{-2\eps}\ddbar\abs{z}^2
  \]
  on $\Omega$ in the sense of currents.  Then for all $1\leq q\leq n$ and $0\leq s<\frac{DF(\Omega)}{2}$ we have the estimate
  \begin{equation}
  \label{eq:basic_estimate_with_weights}
    \norm{\delta^{s-\eps}u}_{L^2(\Omega)}\leq k_{s,q,C,\Omega}\left(\norm{\delta^s\dbar^* u}_{L^2(\Omega)}+\norm{\delta^s\dbar u}_{L^2(\Omega)}\right),
  \end{equation}
  whenever $u\in L^2_{(0,q)}(\Omega)\cap\dom\dbar\cap\dom\dbar^*$ for some constant $k_{s,q,C,\Omega}>0$.
\end{prop}

\begin{proof}
Let $\Omega=\bigcup_j\Omega_j$ where $\{\Omega_j\}_{j\in\mathbb{N}}$ is a sequence of smooth, strictly pseudoconvex domains satisfying $\overline{\Omega_j}\subset\Omega_{j+1}$.  Set $\eta=\sqrt{2 s DF(\Omega)}$, and let $\mu$ be the weight function given by Lemma \ref{lem:diederich_fornaess_weight}.  Note that $\frac{(DF(\Omega)-\eta)C}{2 e DF(\Omega)}=\frac{(\eta-2s)C}{2 e \eta}$ if $\eta\not=0$, and $\frac{(DF(\Omega)-\eta)C}{2 e DF(\Omega)}=\frac{C}{2e}$ if $\eta=0$.  Regularizing by convolution, we may assume that $\mu_j$ is a smooth function on $\overline{\Omega_j}$ such that $\mu_k\rightarrow\mu$ on $\overline\Omega_j$ when $k\geq j$, $i\ddbar\mu_j\geq i \frac{(\eta-2s)C}{9 \eta}(-\mu_j)\delta^{-2\eps}\ddbar\abs{z}^2$ on $\Omega_j$, and $\delta^\eta\leq-\mu_j\leq k_\eta\delta^\eta$ on $\overline\Omega_j$.

\begin{rem}
\label{rem:Kohn_results}
Kohn's work in \cite{Koh78,Koh79} implies that the Bergman Projection $P_{q,j}$ and $\dbar$-Neumann operator $N_{q,j}$ exist on $\Omega_j$ and preserve $C^\infty_{(0,q)}(\overline\Omega_j)$.  Since $\mu_j$ is bounded away from zero on $\overline\Omega_j$, the $L^2$ norm weighted by a power of $-\mu_j$ is equivalent to the unweighted $L^2$ norm for any $s\in\mathbb{R}$.  Hence, on $\Omega_j$ we will prove a priori estimates for forms with coefficients in $C^\infty(\overline\Omega_j)$ and immediately obtain estimates for forms with coefficients in $L^2(\Omega_j)$.
\end{rem}

We recall the `twisted' Kohn-Morrey-H\"ormander formula (see Proposition 2.4 in \cite{Str10}).  If $u\in C^1_{(0,q)}(\overline\Omega_j)\cap\dom\dbar^*$ with $1\leq q\leq n$ and $a\in C^2(\overline\Omega_j)$, then
\begin{multline}
\label{eq:basic_estimate_1}
  \norm{\sqrt{a}\dbar u}^2_{L^2(\Omega_j)}+\norm{\sqrt{a}\dbar^* u}^2_{L^2(\Omega_j)}\geq 2\re\left({\sum_K}'\sum_{\ell=1}^n u_{\ell K}\frac{\partial a}{\partial z_\ell}d\bar z_K,\dbar^* u\right)_{L^2(\Omega_j)}\\
  +{\sum_K}'\sum_{\ell,k=1}^n\int_{\Omega_j}\left(-\frac{\partial^2 a}{\partial z_\ell\partial\bar z_k}\right)u_{\ell K}\bar u_{k K}dV.
\end{multline}
Here we have set $\varphi=0$ and omitted both the boundary term, which is positive on pseudoconvex domains, and the gradient term, which is also positive.  We set $t=\frac{2s}{\eta}$ when $\eta\neq 0$, and $t=\frac{1}{2}$ when $\eta=0$.  We set $a=(-\mu_j)^t$.  Then
\[
  \frac{\partial a}{\partial z_\ell}=-t(-\mu_j)^{t-1}\frac{\partial\mu_j}{\partial z_\ell},
\]
and
\[
  \frac{\partial^2 a}{\partial z_\ell\partial\bar z_k}=-ta(-\mu_j)^{-1}\frac{\partial^2\mu_j}{\partial z_\ell\partial\bar z_k}-t^{-1}(1-t)a^{-1}\frac{\partial a_j}{\partial z_\ell}\frac{\partial a_j}{\partial\bar z_k}.
\]
Hence,
\[
  i\ddbar(-a)\geq i \frac{C}{9}t(1-t)a\delta^{-2\eps}\ddbar\abs{z}^2+it^{-1}(1-t) a^{-1}\partial a\wedge\dbar a.
\]
Substituting in \eqref{eq:basic_estimate_1}, we have
\begin{multline*}
  \norm{\sqrt{a}\dbar u}^2_{L^2(\Omega_j)}+\norm{\sqrt{a}\dbar^* u}^2_{L^2(\Omega_j)}\geq 2\re\left({\sum_K}'\sum_{\ell=1}^n u_{\ell K}\frac{\partial a}{\partial z_\ell}d\bar z_K,\dbar^* u\right)_{L^2(\Omega_j)}\\
  +\frac{Cq}{9}t(1-t)\norm{\sqrt{a}\delta^{-\eps}u}^2_{L^2(\Omega_j)}+t^{-1}(1-t)\norm{a^{-1/2}{\sum_K}'\sum_{\ell=1}^n u_{\ell K}\frac{\partial a}{\partial z_\ell}d\bar z_K}^2_{L^2(\Omega_j)}.
\end{multline*}
The Cauchy-Schwarz inequality gives us
\[
  \norm{\sqrt{a}\dbar u}^2_{L^2(\Omega_j)}+(1-t)^{-1}\norm{\sqrt{a}\dbar^* u}^2_{L^2(\Omega_j)}\geq \frac{Cq}{9}t(1-t)\norm{\sqrt{a}\delta^{-\eps}u}^2_{L^2(\Omega_j)}.
\]
If we let $u=\dbar N_{q-1,j} f$ for $f\in C^\infty_{(0,q-1)}(\overline\Omega_j)$, we have
\[
  (1-t)^{-1}\norm{\sqrt{a}(I-P_{q-1,j})f}^2_{L^2(\Omega_j)}\geq \frac{Cq}{9}t(1-t)\norm{\sqrt{a}\delta^{-\eps}\dbar N_{q-1,j} f}^2_{L^2(\Omega_j)},
\]
and if we let $u=\dbar^* N_{q+1,j} g$ for $g\in C^\infty_{(0,q+1)}(\overline\Omega_j)$, we have
\[
  \norm{\sqrt{a}P_{q+1,j}g}^2_{L^2(\Omega_j)}\geq \frac{Cq}{9}t(1-t)\norm{\sqrt{a}\delta^{-\eps}\dbar N_{q+1,j} g}^2_{L^2(\Omega_j)}.
\]
By using Remark \ref{rem:Kohn_results} we can remove the smoothness assumptions on $f$ and $g$.  If we assume that $f$ and $g$ are restrictions of forms with coefficients in $L^2(\Omega,\sqrt{a})$, then we may extend $\dbar N_{q-1,j}f$, $(I-P_{q-1,j})f$, $\dbar^* N_{q+1,j}g$, and $P_{q+1,j}g$ to be zero on $\Omega\backslash\Omega_j$.  From Berndtsson and Charpentier's Theorem \eqref{thm:BeCh00}, we know that the Bergman Projection is continuous in $L^2(\Omega,a^{-1})$.  Since the Bergman Projection is self-adjoint, we also know that it is continuous in $L^2(\Omega,a)$.  Using the same techniques as the proof of Theorem 1.2 in \cite{MiSh98} (see also 2.10 in \cite{Str10}), we can extract subsequences converging weakly to $\dbar N_{q-1} f$, $(I-P_{q-1})f$, $\dbar^* N_{q+1}g$, and $P_{q+1}g$ in $L^2(\Omega,a)$, so we obtain
\begin{equation}
\label{eq:dbar_N_f_estimate}
  \norm{\sqrt{a}\delta^{-\eps}\dbar N_{q-1} f}_{L^2(\Omega)}\leq \frac{3}{(1-t)\sqrt{tCq}}\norm{\sqrt{a}(I-P_{q-1})f}_{L^2(\Omega)},
\end{equation}
for all $f\in L^2_{(0,q-1)}(\Omega)$ and
\begin{equation}
\label{eq:dbar_star_N_f_estimate}
  \norm{\sqrt{a}\delta^{-\eps}\dbar N_{q+1} g}_{L^2(\Omega)}\leq \frac{3}{\sqrt{t(1-t)Cq}}\norm{\sqrt{a}P_{q+1}g}_{L^2(\Omega)}.
\end{equation}
for all $g\in L^2_{(0,q+1)}(\Omega)$.

Now, for $u\in L^2_{(0,q)}(\Omega)\cap\dom\dbar\cap\dom\dbar^*$, we set $g=\dbar u$ and $f=\dbar^* u$.  Since we have the orthogonal decomposition $u=\dbar^* N_{q+1}g+\dbar N_{q-1}f$, we may use  \eqref{eq:dbar_N_f_estimate} and \eqref{eq:dbar_star_N_f_estimate} to obtain
\[
  \norm{\sqrt{a}\delta^{-\eps}u}_{L^2(\Omega)}\leq\frac{3}{(1-t)\sqrt{tCq}}\norm{\sqrt{a}\dbar^* u}_{L^2(\Omega)}+\frac{3}{\sqrt{t(1-t)Cq}}\norm{\sqrt{a}\dbar u}_{L^2(\Omega)},
\]
from which \eqref{eq:basic_estimate_with_weights} follows.

\end{proof}

\section{Proof of Main Result}

Fix $p>2$ satisfying \eqref{eq:p_range}.  Let $u\in L^2_{(0,q)}(\Omega)\cap\dom\dbar\cap\dom\dbar^*$ for $1\leq q\leq n$.  Then $\alpha=(\dbar\dbar^*+\vartheta\dbar)u$ is a distribution in the $L^2$ Sobolev space $W^{-1}_{(0,q)}(\Omega)$.  Using the $L^2$ theory for the Laplacian we may find $\tilde u$ in $W^1_{(0,q)}(\Omega)$ satisfying $\alpha=(\dbar\dbar^*+\vartheta\dbar)\tilde u$ with vanishing boundary trace.  Furthermore,
\begin{equation}
\label{eq:tilde_u_estimate}
  \norm{\tilde{u}}_{W^1(\Omega)}\lesssim\norm{\alpha}_{W^{-1}(\Omega)}\lesssim\norm{\dbar u}_{L^2(\Omega)}+\norm{\dbar^* u}_{L^2(\Omega)}.
\end{equation}
The Sobolev Embedding Theorem gives us $\tilde{u}\in L^p_{(0,q)}(\Omega)$ whenever $p<\frac{2n}{n-2}$. Hence, it is only necessary to estimate $h=u-\tilde{u}$.  Note that $h$ has harmonic coefficients and satisfies $h\in\dom\dbar\cap\dom\dbar^*$.  By \eqref{eq:Lp_norm_bound_1}, we have
\[
  \norm{h}_{L^p(\Omega)}\lesssim\norm{\delta^{-n(p-2)/p}h}_{L^2(\Omega)}
\]
When $\eps>\frac{n(p-2)}{p}$, we may use \eqref{eq:basic_estimate_with_weights} with $s=0$ to obtain
\[
  \norm{h}_{L^p(\Omega)}\lesssim\norm{\dbar h}_{L^2(\Omega)}+\norm{\dbar^* h}_{L^2(\Omega)}
\]
Substituting $h=u-\tilde{u}$ in the upper bound, we have
\[
  \norm{h}_{L^p(\Omega)}\lesssim\norm{\dbar u}_{L^2(\Omega)}+\norm{\dbar^* u}_{L^2(\Omega)}+\norm{\tilde u}_{W^1(\Omega)},
\]
but then \eqref{eq:tilde_u_estimate} gives us
\[
  \norm{h}_{L^p(\Omega)}\lesssim\norm{\dbar u}_{L^2(\Omega)}+\norm{\dbar^* u}_{L^2(\Omega)}.
\]
Since we have already shown that $\norm{\tilde{u}}_{L^p(\Omega)}$ shares this upper bound, we conclude
\begin{equation}
\label{eq:L_p_basic_estimate}
  \norm{u}_{L^p(\Omega)}\lesssim\norm{\dbar u}_{L^2(\Omega)}+\norm{\dbar^* u}_{L^2(\Omega)}
\end{equation}
for all $p>2$ in the range given by \eqref{eq:p_range} and $u\in\dom\dbar\cap\dom\dbar^*$.  Note that this is an actual estimate, so it is not necessary to assume a priori that $u\in L^p_{(0,q)}(\Omega)$.

Since \eqref{eq:L_p_basic_estimate} is an actual estimate, it is easy to prove estimates for our solution operators.  If $f\in L^2_{0,q-1}(\Omega)$, we may set $u=\dbar N_{q-1}f$ and substitute in \eqref{eq:L_p_basic_estimate} to obtain
\[
  \norm{\dbar N_{q-1}f}_{L^p(\Omega)}\lesssim\norm{(I-P_{q-1})f}_{L^2(\Omega)}\leq\norm{f}_{L^2(\Omega)}.
\]
Similarly, for $g\in L^2_{0,q+1}(\Omega)$, we set $u=\dbar^* N_{q+1}g$ and use \eqref{eq:L_p_basic_estimate} to obtain
\[
  \norm{\dbar^* N_{q+1}g}_{L^p(\Omega)}\lesssim\norm{P_{q+1}g}_{L^2(\Omega)}\leq\norm{g}_{L^2(\Omega)}.
\]
Let $p^*=\frac{p}{p-1}$ denote the conjugate exponent for $p$, so that for $u\in L^{p^*}_{0,q}(\Omega)$ we immediately obtain the dual estimates
\begin{align*}
  \norm{\dbar N_{q}u}_{L^2(\Omega)}&\lesssim\norm{u}_{L^{p^*}(\Omega)},\\
  \norm{\dbar^* N_{q}u}_{L^2(\Omega)}&\lesssim\norm{u}_{L^{p^*}(\Omega)}.
\end{align*}
Using Range's formula $N_q=(\dbar^*N_{q+1})(\dbar N_q)+(\dbar N_{q-1})(\dbar^* N_q)$, we now have the necessary estimates to prove
\[
  \norm{N_q u}_{L^p(\Omega)}\lesssim\norm{u}_{L^{p^*}(\Omega)}
\]
for all $1\leq q\leq n$.

Since $\dbar^* N_{q+1}$ is the canonical solution operator to the inhomogeneous Cauchy-Riemann equations, we are only missing a solution operator when $q=0$.  We will study this by examining the dual operator $(\dbar^* N_1)^*$, which is equal to $N_1\dbar$ on $\dom\dbar$.

Fix $p$ satisfying \eqref{eq:p_range_2}.  We wish to show that $(\dbar^* N_1)^*$ is continuous from the dual space $L^{p/(p-1)}(\Omega)$ to $L^2(\Omega)$.  We easily check that $(\dbar^* N_1)^*$ maps into $\dom\dbar^*$, and $\dbar^*(\dbar^* N_1)^*=I-P_0$.  Note that \eqref{eq:p_range_2} is equivalent to $0<\frac{n(p-2)}{p}<\min\left\{\frac{DF(\Omega)}{4n},\frac{\eps}{2n}\right\}$.  Fix $p'$ satisfying $0<\frac{n(p-2)}{p}<\frac{p'-2}{2p'}<\min\left\{\frac{DF(\Omega)}{4n},\frac{\eps}{2n}\right\}$. From \eqref{eq:Bergman_projection_estimate}, we see that
\[
  \norm{(\dbar^* N_1)^* g}_{L^2(\Omega)}\lesssim\norm{g}_{L^{p/(p-1)}(\Omega)}
\]
will follow from
\[
  \norm{(\dbar^* N_1)^* g}_{L^2(\Omega)}\lesssim\norm{\dbar^*(\dbar^* N_1)^* g}_{L^{p'/(p'-1)}(\Omega)},
\]
so it will suffice to prove
\[
  \norm{u}_{L^2(\Omega)}\lesssim\norm{\dbar^*u}_{L^{p'/(p'-1)}(\Omega)}
\]
for all $u\in L^2_{(0,1)}(\Omega)\cap\dom\dbar^*\cap\ker\dbar$ such that $\dbar^* u\in L^{p'/(p'-1)}(\Omega)$.

Let $p^*=\frac{p'}{p'-1}$ denote the conjugate exponent to $p'$, and note that $\frac{2-p^*}{p^*}=\frac{p'-2}{p'}$.  Let $u\in\dom\dbar\cap\dom\dbar^*$ satisfy $\dbar u\in L^{p^*}_{(0,2)}(\Omega)$ and $\dbar^* u\in L^{p^*}_{(0,0)}(\Omega)$.  Then $\alpha=(\dbar\dbar^*+\vartheta\dbar)u$ is a distribution with coefficients in the $L^{p^*}$ Sobolev space $L^{p^*}_{-1}(\Omega)$.  When $n=1$, $\frac{p'-2}{2p'}<\frac{1}{4}$, so $p'<4$.  When $n\geq 2$, $\frac{p'-2}{2p'}<\frac{1}{8}$, so $p'<\frac{8}{3}<3$.  Hence, we may use Theorem 0.5 in \cite{JeKe95} to find $\tilde u$ with coefficients in $L^{p^*}_1(\Omega)$ satisfying $\alpha=(\dbar\dbar^*+\vartheta\dbar)\tilde u$ with vanishing boundary trace.  Furthermore,
\[
  \norm{\tilde{u}}_{L^{p^*}_1(\Omega)}\lesssim\norm{\alpha}_{L^{p^*}_{-1}(\Omega)}\lesssim\norm{\dbar u}_{L^{p^*}(\Omega)}+\norm{\dbar^* u}_{L^{p^*}(\Omega)}.
\]
Once again, $h=u-\tilde{u}$ has harmonic coefficients and satisfies $h\in\dom\dbar\cap\dom\dbar^*$, $\dbar h\in L^{p^*}_{(0,2)}(\Omega)$, and $\dbar^* h\in L^{p^*}_{(0,0)}(\Omega)$.  Furthermore, $\dbar h$ and $\dbar^* h$ will also have harmonic coefficients.

Set $s=\frac{n(2-p^*)}{p^*}$, so since $s<\eps$ we obtain from Proposition \ref{prop:basic_estimate}
\[
  \norm{h}_{L^2(\Omega)}\lesssim\norm{\delta^s\dbar^* h}_{L^2(\Omega)}+\norm{\delta^s\dbar h}_{L^2(\Omega)}.
\]
We may use \eqref{eq:Lp_norm_bound_2} to show
\[
  \norm{h}_{L^2(\Omega)}\lesssim\norm{\dbar^* h}_{L^{p^*}(\Omega)}+\norm{\dbar h}_{L^{p^*}(\Omega)}.
\]
As before, this suffices to prove
\[
  \norm{u}_{L^2(\Omega)}\lesssim\norm{\dbar^* u}_{L^{p^*}(\Omega)}+\norm{\dbar u}_{L^{p^*}(\Omega)}.
\]
for any $u\in L^2_{(0,1)}(\Omega)\cap\dom\dbar\cap\dom\dbar^*$ satisfying $\norm{\dbar^* u}_{L^{p^*}(\Omega)}+\norm{\dbar u}_{L^{p^*}(\Omega)}<\infty$.  We have already explained why this implies
\[
  \norm{(\dbar^* N_1)^* g}_{L^2(\Omega)}\lesssim\norm{g}_{L^{p/(p-1)}(\Omega)},
\]
for all $g\in L^{p/(p-1)}(\Omega)$, and hence
\[
  \norm{\dbar^* N_1 f}_{L^p(\Omega)}\lesssim\norm{f}_{L^2(\Omega)}
\]
for all $f\in L^2_{(0,1)}(\Omega)$.


\end{document}